\documentclass[11pt]{amsart}
\usepackage{amscd, amssymb, mathrsfs, mathabx, tikz-cd, amsthm, thmtools}

\author[Oh]{Jeongseok Oh}
\address{Department of Mathematical Sciences and Research Institute of Mathematics\\ 
Seoul National University\\
Seoul 08826\\ 
Korea}
\email{jeongseok@snu.ac.kr}

\usepackage{hyperref}

\input{xy}
\xyoption{all}

\newtheorem{Thm}{Theorem}[section]

\newtheorem{Def/Thm}[Thm]{Definition/Theorem}
\newtheorem{Cor}[Thm]{Corollary}
\newtheorem{Lemma}[Thm]{Lemma}

\theoremstyle{definition}
\newtheorem{Rmk}[Thm]{Remark}
\newtheorem{example}{Example}

\numberwithin{equation}{section}

\newcommand{\ot }{\otimes}
\newcommand{\ra }{\rightarrow}

\newcommand{\Hom }{{\mathrm{Hom}}}

\newcommand{\rank }{{\mathrm{rank}}}

\newcommand{\cO}{{\mathcal{O}}}

\newcommand{\A}{{\mathbb A}}
\newcommand{\PP }{{\mathbb P}}

\newcommand{\CC }{{\mathbb C}}
\newcommand{\ZZ }{{\mathbb Z}}

\newcommand{\ch}{\mathrm{ch}}

\newcommand{\lan}{\langle}
\newcommand{\ran}{\rangle}

\newcommand{\wY}{\widetilde{Y}}

\newcommand{\wE}{\widetilde{E}}
\newcommand{\wa}{\widetilde{\alpha}}
\newcommand{\wb}{\widetilde{\beta}}

\newcommand\arXiv[1]{\href{http://arxiv.org/abs/#1}{arXiv:#1}}
\newcommand\mathAG[1]{\href{http://arxiv.org/abs/math/#1}{math.AG/#1}}

\begin{document}

\title[Multiplicative Property]{Multiplicative property of localized Chern characters for 2-periodic complexes}

\begin{abstract}
In this paper, we prove the multiplicative property of localized Chern characters. As a direct consequence, a localized Chern character gives rise to a ring homomorphism from the $K$-group of periodic complexes to the bivariant Chow cohomology group.
As an application, we prove the functoriality of Kiem--Li's cosection-localized intersection homomorphisms. 
\end{abstract}

\maketitle

\setcounter{tocdepth}{1}
\tableofcontents

\section{Introduction}
Chern character is a ring homomorphism between $K$-group of locally free sheaves and Chow cohomology group for a scheme. It is known to be one of the most important ingredients to study the Riemann--Roch theorem and may be of other interests.
A localized Chern character\footnote{By pushing objects forward by $X\hookrightarrow Y$, the localized Chern character becomes the usual one, see \cite[Proposition 18.1(a)]{Ful} and \eqref{CoMM} below. This is where the terminology ``localized" comes from.} was introduced by Baum--Fulton--MacPherson for studying the Riemann--Roch for singular schemes \cite{BFM}. Given a closed immersion $X\hookrightarrow Y$, this is an additive homomorphism between the $K$-group of {\em bounded complexes} of locally free sheaves on $Y$ whose cohomologies are supported on $X$ and the bivariant Chow group for the immersion \cite[Chapter 17]{Ful}, cf. \eqref{Add:Hom} below. We consider an immersion of a singular scheme into a smooth space for the Riemann--Roch. 

The construction was extended to the $K$-group of {\em unbounded $2$-periodic complexes} of locally free sheaves on $Y$ which are {\em strictly exact off} $X$ by Polishchuk--Vaintrob. We explain the strict exactness precisely in the following section. They used it to define Witten's top Chern class, which provided an algebraic definition of a virtual fundamental class of a moduli space of spin curves. Although the Riemann--Roch theorem is already one very powerful application of (localized) Chern characters\cite{PV:A}, this implicitly tells us that they can be used to study virtual fundamental classes as well.

Unlike a usual one or a localized one for a bounded complex, the multiplicative property of a localized Chern character for a $2$-periodic complex is not known in general. It is conjectured in the first remark of \cite[Section 2.3]{PV:A}. Meanwhile, even if a localized Chern character for a bounded complex is multiplicative \cite[Example 18.1.5]{Ful}, it is mentioned that ``it would be interesting to find a direct proof of this multiplicative property" there.

The purpose of this article is to prove the multiplicative property. The proof is direct enough so that it can be extended to that for Deligne--Mumford (DM for short) stacks as well. Then we would like to introduce a functorial property which we may encounter when we study virtual fundamental classes as an application of the multiplicative property.

\subsection*{Localized Chern character on a DM stack} Since Polishchuk--Vaintrob constructed localized Chern characters on schemes over $\CC$, we would like to briefly mention how we extend this construction to DM stacks over an arbitrary field ${\bf k}$.

Let $Y$ be a finite type DM stack over ${\bf k}$ having a proper covering by a scheme (see the assumption of \cite[Section 3]{Vis}). Let $A_*(Y)$ denote its Chow group with rational coefficients \cite{G, Vis}. One key fact for the extended construction is that proper pushforwards, flat pullbacks and Gysin pullbacks are well-defined \cite[Proposition 3.7 and Theorem 3.11]{Vis} satisfying expected commutative properties \cite[Lemma 3.9, Theorems 3.12 and 3.13]{Vis}. So the bivariant Chow groups are well-defined (cf. \cite[Section 5]{Vis}). Other useful properties including excision sequence and isomorphicity of bundle pullbacks are studied in \cite{Kr} via \cite[Theorem 2.1.12(ii)]{Kr}. Then we can follow the construction \cite[Chapter 3]{Ful} of Chern classes (cf. \cite[Theorem 2.1.12(viii)]{Kr}) so that we can define Chern characters, Todd classes following \cite[Examples 3.2.3 and 3.2.4]{Ful}. 

Once we have these, we can mimic the construction of localized Chern characters of Polishchuk--Vaintrob \cite[Section 2.2]{PV:A}. Let $E_\bullet$ be a 2-periodic complex on $Y$ 
$$E_\bullet \ : =\ \cdots \xrightarrow{d_-}\, E_+\, \xrightarrow{d_+}\, E_-\, \xrightarrow{d_-}\, E_+\, \xrightarrow{d_+} \cdots$$ 
where $E_+$ sits in even degrees and $E_-$ sits in odd degrees. Here and throughout the paper, a $2$-periodic complex will always mean a $2$-periodic complex of locally free sheaves. We say $E_\bullet$ is {\em strictly exact off a closed substack} $X\hookrightarrow Y$ if it satisfies the following two conditions:
\begin{itemize}
\item $E_\bullet$ is exact on $Y-X$.
\item $\ker d_+$ and $\ker d_-$ are vector bundles on $Y- X$.
\end{itemize}
For a 2-periodic complex $E_\bullet$ on $Y$ strictly exact off $i: X\hookrightarrow Y$, a localized Chern character
\begin{align} \label{Local:Ch}
\mathrm{ch}^Y_X (E_\bullet)\ \in\ A^*(i:X\hookrightarrow Y)
\end{align}
is defined to satisfy
\begin{align} \label{CoMM}
i_*\left(\mathrm{ch}^Y_X (E_\bullet)\right)\ =\ \mathrm{ch}(E_+)  -\mathrm{ch}(E_-)\ \in\ A^*(Y).
\end{align}
We introduce more details in Section \ref{constreview}.

An important remark is that \cite[Lemma 2.1, Propositions 2.2 and 2.3]{PV:A} also work for DM stacks. Essentially it only needs commutativities of pushforwads and pullbacks, excision sequences and isomorphicity of bundle pullbacks which we mentioned above. Then one can follow the proofs in \cite[Chapter I, Proposition 3.4 and Chapter II, Section 2]{BFM} and \cite{PV:A}. For \cite[Proposition 2.3(i)]{PV:A}, we may need the projection formulae, but those for projective morphisms \cite[Proposition 3.2.4(ii) or 3.6.2(iii)]{Kr} are enough. In particular, \cite[Proposition 2.3(iv)]{PV:A} shows $\mathrm{ch}^Y_X$ defines an additive homomorphism \eqref{Add:Hom} below.

\subsection*{Multiplicative property}
Now we introduce the precise statement of the multiplicative property.

\begin{Thm} \label{Multi:Prop}
Let $E^1_\bullet$, $E^2_\bullet$ be 2-periodic complexes on $Y$ strictly exact off $i_1:X_1 \hookrightarrow Y$, $i_2:X_2 \hookrightarrow Y$, respectively. Then the following multiplicative property holds,
\begin{align*}
\mathrm{ch}^Y_{X_1 \cap X_2} (E^1_\bullet \ot E^2_\bullet )\ =\ \mathrm{ch}^{X_1}_{X_1 \cap X_2} ( i^*_1 E^2_\bullet ) \circ \mathrm{ch}^{Y}_{X_1 } ( E^1_\bullet ) .
\end{align*}
\end{Thm}

The tensor product $E^1_\bullet \ot E^2_\bullet$ above is defined by considering $E^i_\bullet$ as $\ZZ/2$-graded sheaves of $\cO_Y$-modules. To see that it is strictly exact off $X_1\cap X_2$, it is enough to show that the tensor product is strictly exact when $E^1_\bullet$ is, without loss of generality, for the index $i=1$ and $2$. It comes from the facts that locally we have isomorphisms $E^1_+\cong \ker d^1_+\oplus \ker d^1_-\cong E^1_-$ of bundles and the differential morphisms are $d^1_-=\mathrm{id}\oplus 0$ and $d^1_+=0\oplus\mathrm{id}$ along the decompositions when $E^1_\bullet$ is strictly exact. 

Theorem \ref{Multi:Prop} is proved in \cite[Proposition 2.3(vi)]{PV:A} for a special case when 
\begin{itemize}
\item $Y$ is the total space of a vector bundle $V$ on $X_1$,
\item $E^1_\bullet$ is the Koszul 2-periodic complex $\Lambda^\bullet\pi^*V^*$, where $\pi:Y\to X_1$ is the bundle projection, which is then strictly exact off $X_1$ and
\item $E^2_\bullet$ is the pullback of a 2-periodic complex on $X_1$ strictly exact off $X=X_1 \cap X_2$.
\end{itemize} 
In \cite[Lemma 2.18]{BKO1}, Sreedhar and the author proved $K$-theoretic version of Theorem \ref{Multi:Prop}.

\smallskip
We define $K^{\ZZ /2, st}_0(Y)_X$ to be an abelian group generated by $2$-periodic complexes of {\em locally free sheaves} on $Y$ strictly exact off $X$ with the relations coming from short exact sequences of complexes of {\em coherent sheaves}.\vspace{-0.5mm} The tensor product defines a ring structure on $K^{\ZZ /2, st}_0(Y)_X$ since it preserves short exact sequences.
As a direct consequence of Theorem \ref{Multi:Prop}, we obtain the following corollary.
\begin{Cor} \label{coradd}
The additive homomorphism
\begin{align} \label{Add:Hom}
\ch^Y_X\ :\ K^{\ZZ /2, st}_0(Y)_X \ \longrightarrow\ A^*(X \xrightarrow{i} Y)
\end{align}
is a ring homomorphism.
\end{Cor}

\subsection*{Functoriality of cosection-localized intersection homomorphisms}
A cosection-localized intersection homomorphism constructed by Kiem--Li \cite{KL1} has become one of the most useful tools to study virtual fundamental classes of moduli spaces. It localizes the Fulton--MacPherson's intersection homomorphism (taking a cycle in a vector bundle to the intersection with the zero section) in the following sense.

Let $M$ be a finite type DM stack and $F$ be a vector bundle on $M$ or a locally free sheaves of $\cO_M$-modules.
Let $\sigma: F \ra \cO_M$ be a homomorphism of $\cO_M$-modules (called a \emph{cosection}).
We denote by $w_\sigma \in \Gamma(\cO_F)$ the function on the total space of $F$ defined by $\sigma$.
Then the zero loci $Z(w_\sigma)$ of the function $w_\sigma$ and $Z(\sigma):=Z(\sigma^*)$ of the section $\sigma^*$ satisfy
\begin{align*}
Z(\sigma)\ \hookrightarrow\ M\ \hookrightarrow\ Z(w_\sigma)\ \hookrightarrow\ F,
\end{align*}
where the immersion in the middle is given by the zero section.
Kiem--Li constructed an intersection homomorphism of degree$=\rank F$ \cite[Proposition 1.3]{KL1}
$$0^!_{F,\sigma}\ :\ A_*(Z(w_\sigma))\ \longrightarrow\ A_{*-\rank F}(Z(\sigma))$$
such that the following diagram commutes 
$$
\xymatrixcolsep{5pc}
\xymatrix{
A_*(Z(w_\sigma)) \ar[r]^-{\text{pushforward}} \ar[d]_-{0^!_{F,\sigma}} & A_*(F) \ar[d]^-{0^!_F} \\
A_{*-\rank F}(Z(\sigma)) \ar[r]_-{\text{pushforward}} & A_{* - \rank F}(M),
}
$$
where $0^!_F$ denotes the Fulton--MacPherson's intersection homomorphism.
The homomorphism $0^!_{F, \sigma}$ is a \emph{cosection-localized intersection homomorphism}.

\begin{Rmk}
Kiem--Li's construction of the cosection-localized intersection homomorphisms \cite[Section 2, particularly Corollary 2.9]{KL1} actually works for an arbitrary field although they started with $\CC$. The place they need $\CC$ is the cone reduction criterion \cite[Lemma 4.4]{KL1} which tells us the Behrend-Fantechi cone is contained in the domain of the homomorphism under a mild assumption so that we can define the localized virtual cycle as its image.
\end{Rmk}

In \cite[Theorem 1.1]{KO1}, Kim and the author proved a cosection-localized intersection homomorphism is equivalent to a localized Chern character up to a Todd class. Using this expression and the multiplicative property of a localized Chern character, we prove the functorial property of cosection-localized intersection homomorphisms.
\begin{Thm} \label{Cosec:Func}
Let $F_1, F_2$ be vector bundles on $M$ and 
$$
\sigma_1\ :\ F_1\ \longrightarrow\ \cO_M,\ \ \sigma_2\ :\ F_2\ \longrightarrow\ \cO_M$$
be cosections. Then the following functoriality holds true
\begin{align*}
0^!_{F_1 \oplus F_2, \sigma_1 \oplus \sigma_2} \ =\ & 0^!_{F_1, \sigma_1} \circ 0^!_{ F_2,  \sigma_2} \\
&:\ A_*\left( Z(w_{\sigma_1}) \times_M Z(w_{\sigma_2}) \right)\ \longrightarrow\ A_{*-r}\left( Z(\sigma_1 \oplus \sigma_2) \right),
\end{align*}
where $r:= \rank F_1 + \rank F_2$ and $0^!_{F_2, \sigma_2}$ is considered to be a homomorphism $A_*\left( Z(w_{\sigma_1}) \times_M Z(w_{\sigma_2}) \right) \ra A_{*-\rank F_2}\left( Z(w_{\sigma_1}) \times_M Z(\sigma_2) \right)$ by its bivariant property.
\end{Thm}
Note that the homomorphism $0^!_{F_1 \oplus F_2, \sigma_1 \oplus \sigma_2}$ is defined through
$$
Z(w_{\sigma_1}) \times_M Z(w_{\sigma_2})\ \hookrightarrow\ Z(w_{\sigma_1 \oplus \sigma_2})\ =\ Z(w_{\sigma_1} + w_{\sigma_2}).
$$
It seems interesting to think if $0^!_{F_1, \sigma_1} \circ 0^!_{ F_2,  \sigma_2}$ can be considered to be an operator defined on $A_*(Z(w_{\sigma_1 \oplus \sigma_2}))$.

\subsection*{Plan}

Here is a plan of the paper.
In Section \ref{proof}, we review a definition of a localized Chern character \eqref{Local:Ch} and prove the multiplicative property (Theorem \ref{Multi:Prop}). The key idea of the proof is to construct a deformation \eqref{xi2} of the $2$-periodic complex $E_\bullet$.
In Section \ref{app}, we prove the functoriality of cosection-localized intersection homomorphisms (Theorem \ref{Cosec:Func}) as an application.

\bigskip

\noindent{\bf Acknowledgments}
The author would like to thank Bumsig Kim for suggesting the author to think about the problem.
He thank Feng Qu for pointing out a mistake in the earlier draft and advices.
He also thank Alexander Polishchuk and the kind referee for helpful advices and comments.

J. O. was supported by the KIAS individual grant MG063002, the New Faculty Startup Fund from
Seoul National University and the National Research Foundation of Korea (NRF) grant funded by the Korean government (MSIT)(RS-2024-00339364).

\section{Multiplicative property of localized Chern characters} \label{proof}
 
\subsection{Construction of localized Chern characters \eqref{Local:Ch}} \label{constreview}
We review the construction in \cite[Section 2]{PV:A} of localized Chern characters \eqref{Local:Ch} briefly, but for a DM stack $Y$ over ${\bf k}$. As we have mentioned, the idea is based on that in \cite{BFM} for bounded complexes in \cite[Chapter 18]{Ful}. So we refer the results in \cite[Chapter 18]{Ful} together with those in \cite[Section 2]{PV:A}.

\subsection*{Construction}
When $X=Y$, we define
\begin{align} \label{X=Y}
\mathrm{ch}^Y_X(E_\bullet)\ :=\ \mathrm{ch}(E_+) - \mathrm{ch}(E_-).
\end{align}
Now we assume that $X \neq Y$.
To define the class \eqref{Local:Ch}, we would like to construct $\ch^Y_X(E_\bullet) \cap [V] \in A_*(X)$ for an integral closed substack $V$ of $Y$.

Let $Gr:=Gr(\rank E_+,E_+ \oplus E_-)$ be the Grassmannian over $Y$ and $G:= Gr \times_Y Gr$ be the product of the Grassmannian.
Using the fact that $\rank E_+ =\rank E_-$ when $X\neq Y$, we define a locally closed immersion
\begin{align}\label{PHI}
\phi\ :\ Y \times \A^1 \ \rightharpoonup\ G \times \PP^1, \ \ \
 (y, \lambda)\ \mapsto\ (\, \Gamma_{\lambda d_+(y)},\, \Gamma_{\lambda d_-(y)},\, \lambda),
\end{align}
where $\Gamma$ denotes the graph and $d_\pm(y): E_\pm |_y \ra E_\mp|_y$ is the induced homomorphism at fiber. This $\phi$ parametrizes the two $\rank E_+$ subbundles 
\begin{align}\label{f3}
\ker(E_+\oplus E_-\xrightarrow{\lambda d_+-\mathrm{id}} E_-),\ \ker(E_+\oplus E_-\xrightarrow{\mathrm{id}-\lambda d_-} E_+)\ \subset\ E_+\oplus E_-
\end{align}
over $Y\times\mathbb{A}^1$, hence it is a morphism of stacks. Since it is a closed immersion into $G\times\mathbb{A}^1$, it is a locally closed immersion into $G\times\PP^1$. We use the notation $\rightharpoonup$ for locally closed immersions instead of $\hookrightarrow$ to distinguish them to closed immersions.
Let $Y_\infty \subset G$ be the restriction to $\infty \in \PP^1$ of the scheme theoretic image of $\phi$ which is the closure $\overline{\phi(Y \times \mathbb{A}^1)}$ set-theoretically. 

Recall that $\rank(\ker d_\pm)$ can be considered as locally constant functions on $Y-X$. Hence they can be extended to locally constant functions on $Y$.
Using these, we define
$$
H\ :=\ Gr(\rank(\ker d_+), E_+) \times_Y Gr(\rank(\ker d_-), E_-)
$$ 
which is another product of Grassmannians over $Y$. The direct sum gives rise to an immersion $H \hookrightarrow Gr \hookrightarrow G$, where the second one is the diagonal morphism. 
On $Y-X$, we have isomorphisms $E_+\cong \ker d_+\oplus \ker d_-\cong E_-$ of bundles locally and the differential morphisms are $d_-=\mathrm{id}\oplus 0$ and $d_+=0\oplus\mathrm{id}$. Using these, the bundles \eqref{f3} defining $\phi$ identify with the following $\rank=\rank E_+$ bundles locally on $(Y-X)\times\mathbb{A}^1$
\begin{align*}
\mathrm{im}\left(\ker d_+\oplus \ker d_-\xrightarrow{(\mathrm{id}\oplus\mathrm{id})\oplus (0\oplus\lambda\mathrm{id})} (\ker d_+\oplus \ker d_-)\oplus (\ker d_+\oplus \ker d_-)\right)\\
\text{and}\qquad\qquad\qquad\qquad\qquad\qquad\qquad\qquad \\
\mathrm{im}\left(\ker d_+\oplus \ker d_-\xrightarrow{(\lambda\mathrm{id}\oplus 0)\oplus (\mathrm{id}\oplus\mathrm{id})} (\ker d_+\oplus \ker d_-)\oplus (\ker d_+\oplus \ker d_-)\right).
\end{align*}
These are extended to bundles on $(Y-X)\times\PP^1$
\begin{align} \label{hahahaha}
\nonumber \mathrm{im}\left(\ker d_+\oplus \ker d_-\xrightarrow{(\mathrm{id}\oplus\mu\mathrm{id})\oplus (0\oplus\nu\mathrm{id})} (\ker d_+\oplus \ker d_-)\oplus (\ker d_+\oplus \ker d_-)\right)\\
\text{and}\qquad\qquad\qquad\qquad\qquad\qquad\qquad\qquad \\
\mathrm{im}\left(\ker d_+\oplus \ker d_-\xrightarrow{(\nu\mathrm{id}\oplus 0)\oplus (\mu\mathrm{id}\oplus\mathrm{id})} (\ker d_+\oplus \ker d_-)\oplus (\ker d_+\oplus \ker d_-)\right), \nonumber
\end{align}
where $(\nu,\mu)\in\PP^1$ is the homogeneous coordinates on $\PP^1$ with $\lambda=\nu/\mu$. Although we have used the local decompositions of $E_+\cong \ker d_+\oplus \ker d_-\cong E_-$, the bundles \eqref{hahahaha} are global since $\phi$ is globally defined. So it defines a closed immersion $(Y-X)\times\PP^1\hookrightarrow G|_{Y-X}\times\PP^1$. When $\lambda=\infty$, i.e. $(\nu,\mu)=(1,0)$, both bundles on $(Y-X)\times\{\infty\}$ are 
$$
(\ker d_+\oplus 0)\oplus (0\oplus \ker d_-)\ \subset\ (\ker d_+\oplus \ker d_-)\oplus (\ker d_+\oplus \ker d_-),
$$
which means $Y_\infty \times_Y (Y-X)$ is $Y-X$ and the restriction $Y_\infty \times_Y (Y-X) \hookrightarrow G|_{Y-X}$ of the closed immersion $\phi|_\infty:Y_\infty\hookrightarrow G$ to $Y-X\subset Y$ factors through
\begin{align*} 
Y_\infty \times_Y (Y-X) \ \hookrightarrow\ H|_{Y-X} \ \hookrightarrow\  G|_{Y-X}.
\end{align*}
This tells us that $Y_\infty$ is closely immersed into $G|_X \cup H$. 

Let $\widetilde{V}$ be the scheme theoretic image of $\phi:V \times \mathbb{A}^1 \hookrightarrow G \times \PP^1$. We define $[V_\infty]$ as its Gysin pullback image by the regular immersion $\{\infty\} \hookrightarrow \PP^1$,
\begin{align}\label{Vinfty}
[V_\infty] \ := \ \infty^!\,[\widetilde{V}] \ \in\ A_{\dim V}(Y_\infty),
\end{align}
which lies in $Y_\infty$.
Since $Y_\infty$ is embedded in the union $Y_\infty \subset G|_X \cup H$, we can choose a component $[V_{X,\infty}] \in A_{\dim V}(G|_X \cap Y_\infty)$ of $[V_\infty]$ lying in $G|_X$. Then the difference lies in the other component $H$. More precisely, we have
\begin{align} \label{sitH}
[V_\infty] - j_*[V_{X,\infty}] \ \in \ \text{image}\left( A_{\dim V}(H \cap Y_\infty) \ra A_{\dim V} (Y_\infty) \right),
\end{align}
where $j: G|_X \cap Y_\infty \hookrightarrow Y_\infty$ is the closed immersion.
Then $\ch^Y_X(E_\bullet) \cap [V]$ is defined to be
\begin{align} \label{Def:LcCh}
\ch^Y_X(E_\bullet) \cap [V]\ :=\ p_{X*}\left( \left( \ch(\xi_+) - \ch(\xi_-) \right) \cap [V_{X,\infty}]  \right)
\end{align}
where $p_X: G|_X \ra X$ is the projection morphism and $\xi_\pm$ are the tautological bundles on each component of $G$. Note that the restrictions $\xi_+|_H$ and $\xi_-|_H$ are isomorphic as $H$ factors through the diagonal morphism $Gr\hookrightarrow G$.

\subsection*{Specialization map}
Let $X_\infty := Y_\infty \times_Y X$. Unfortunately, the specialization morphism
$$
A_*(Y) \ \longrightarrow\ A_*(X_\infty), \ \ \ [V]\ \longmapsto\ [V_{X,\infty}].
$$
constructed above may not be well-defined. We only know $[V_{X,\infty}]|_{X_\infty-H}=[V_\infty]|_{X_\infty-H}$ is well-defined. Hence the specialization morphism is well-defined only up to a class in $H$.
Letting $B_*(X_\infty) := A_*(X_\infty) / A_*(X_\infty \cap H)$ be the quotient (which is isomorphic to $A_*(X_\infty \backslash H)$ by the excision sequence \cite[Proposition 2.3.6]{Kr}), we obtain a well-defined specialization morphism
\begin{align} \label{mod:sp}
sp\ :\ A_*(Y)\ \longrightarrow\ B_*(X_\infty), \ \ \ [V]\ \longmapsto\ [V_{X,\infty}].
\end{align}
Then we can rewrite \eqref{Def:LcCh} by
\begin{align} \label{Alter:deC}
\ch^Y_X(E_\bullet) \cap - \ =\ p_{X*}\left( \left( \ch(\xi_+) - \ch(\xi_-) \right) \cap sp\,(-)  \right).
\end{align}

\begin{Rmk}
We defined a homomorphism $\ch^Y_X(E_\bullet):A_*(Y)\to A_*(X)$ in \eqref{Alter:deC}.
Similarly for any morphism $Y' \ra Y$, we obtain $A_*(Y')\to A_*(Y'\times_Y X)$ by doing the same construction using the pullback of $E_\bullet$ to $Y'$, so that we get a bivariant class
\begin{align*}
\ch^Y_X(E_\bullet)\ \in\ A^*(X \hookrightarrow Y).
\end{align*}
\end{Rmk}

\subsection{Alternative description of \eqref{Alter:deC}}
We keep assuming $X \neq Y$ so that we have $ \rank E_+ = \rank E_-$. In this section, we provide an equivalent description of a localized Chern character which seems to behave better in deformations. The key point is to view $\xi_\pm$ as a part of a complex. This is also the key idea of the proof of the multiplicative property.

\subsection*{Periodic complex structure on $\xi_\pm$}
Here we would like to explain that the tautological bundles $\xi_\pm$ form a 2-periodic complex
\begin{align} \label{xi2}
\xymatrix{
\xi_\bullet \ :=\  \{\,  \xi_+  \ar@<1ex>[r]  ^-{\delta_+}
&   \ar@<1ex>[l]^-{\delta_-} \xi_-  \}
}
\end{align}
on the scheme theoretic image $\widetilde{Y}\subset G\times (\PP^1\smallsetminus\{0\})$ of $\phi|_{Y\times(\mathbb{A}^1\smallsetminus\{0\})}$ \eqref{PHI}. We will prove it is strictly exact off its restriction $\widetilde{X} := \widetilde{Y} \times_Y X$ to $X$. 

We start with an exact 2-periodic complex on $G$
\begin{align} \label{BIG2p}
\xymatrix{
\{ \, E_+ \oplus E_-\ \ar@<1ex>[r]  ^-{pr_-}
&  \ar@<1ex>[l]^-{pr_+}
\ E_+ \oplus E_- \}
}
\end{align}
where $pr_\pm$ are the projection morphisms. Here and after, we omit the notations of pullbacks for bundles by abuse of notations.
Let $\alpha_\pm \in \Hom (\xi_\pm, (E_+ \oplus E_-)/ \xi_\mp)$ be the homomorphisms defined by 
$$
\alpha_\pm\ :\ \xi_\pm\, \subset\, E_+ \oplus E_-\, \xrightarrow{pr_\mp}\, E_+ \oplus E_-\, \longrightarrow\, (E_+ \oplus E_-)/ \xi_\mp.
$$
Then on a closed substack $(\alpha_+)^{-1}(0) \cap (\alpha_-)^{-1}(0)\subset G$, 
the complex \eqref{BIG2p} restricts to a 2-periodic complex
$$\xymatrix{
\xi_\bullet\ :=\  \{\,  \xi_+  \ar@<1ex>[r]  ^-{\delta_+}
&   \ar@<1ex>[l]^-{\delta_-} \xi_-  \},
}
$$
where $\delta_\pm$ denote the factored morphisms of the compositions,
$$
\xymatrix@R=1mm@C=8mm{
\xi_\pm \ar@{-->}[r]^-{\delta_\pm}\ar@{^(->}[dd] & \xi_\mp \ar@{^(->}[dd] \\
&&\hspace{-7mm}\text{on }\ (\alpha_+)^{-1}(0) \cap (\alpha_-)^{-1}(0). \\
E_+\oplus E_- \ar[r]_-{pr_\mp} &E_+\oplus E_-
}
$$
One can check $\phi\left(Y \times (\mathbb{A}^1\smallsetminus\{0\})\right) \subset \left((\alpha_+)^{-1}(0) \cap (\alpha_-)^{-1}(0) \right) \times (\A^1 \smallsetminus \{0\})$ so that we have 
\begin{align*} 
\widetilde{Y}\ \subset\ \left((\alpha_+)^{-1}(0) \cap (\alpha_-)^{-1}(0) \right) \times (\PP^1 \smallsetminus \{0\}).
\end{align*}
Hence the 2-periodic complex $\xi_\bullet$ is defined on $\widetilde{Y}$.

It remains to show that $\xi_\bullet$ is strictly exact off $\widetilde{X}$.
By the description \eqref{hahahaha} of tautological bundles $\xi_\pm$, one can check that $\ker d_+\oplus 0\subset E_+\oplus E_-$ and $0\oplus \ker d_-\subset E_+\oplus E_-$\vspace{-0.3mm} are contained in the restrictions of the bundles $\xi_+|_{\widetilde{Y}\smallsetminus\widetilde{X}}$ and $\xi_-|_{\widetilde{Y}\smallsetminus\widetilde{X}}$ to \vspace{-0.7mm}$\widetilde{Y}\smallsetminus\widetilde{X}$; and are isomorphic to $\ker\delta_+=\mathrm{im}\delta_-$ and $\ker\delta_-=\mathrm{im}\delta_+$, respectively. Hence $\xi_\bullet$ is strictly exact off $\widetilde{X}$.

One observation we are going to use later is that the restriction of $(\widetilde{Y}, \xi_\bullet)$\vspace{-0.3mm} to $\widetilde{Y}|_\lambda=\phi\;(Y\times \{\lambda\})\cong Y$, $\lambda\neq 0,\infty$, is isomorphic to $(Y, E_\bullet)$,
$$
\xymatrix@R=1mm@C=15mm{
\xi_\pm|_\lambda \ar[r]^-{\delta_\pm}\ar[dd]^-{\lambda pr_\pm}_-\cong & \xi_\mp|_\lambda \ar[dd]^-{\lambda pr_\mp}_-\cong \\
&&\hspace{-8mm}\text{on }\ \widetilde{Y}|_\lambda=\phi\;(Y\times \{\lambda\})\cong Y. \\
E_\pm \ar[r]_-{d_\pm} &E_\mp
}
$$

\subsection*{Alternative description of the localized Chern character \eqref{Alter:deC}}
Let
$$
sp_A\ :\ A_*(Y)\ \longrightarrow\ A_*(Y_\infty), \ \ \ [V]\ \longmapsto\ [V_{\infty}]
$$
be the specialization morphism taking an integral closed substack $V$ to $V_\infty$ constructed in \eqref{Vinfty}. Beware that it is different from $sp$ in \eqref{mod:sp}.

\begin{Lemma} \label{rein}
We obtain an equivalence 
\begin{align*}
\ch^Y_X(E_\bullet) \cap -\ =\ p_{X*} \left( \ch^{Y_\infty}_{X_\infty}(\xi_\bullet|_\infty) \cap sp_A(-) \right) .
\end{align*}
\end{Lemma}

\begin{proof}
By definition of $\ch^Y_X(E_\bullet)$ in \eqref{Alter:deC} and \eqref{X=Y}, we obtain
$$
\ch^Y_X(E_\bullet) \cap -\ =\ p_{X*} \left( \ch^{X_\infty}_{X_\infty}(\xi_\bullet|_{X_\infty}) \cap sp\, (-) \right).
$$
When $X=Y$ it follows from the $\mathbb{A}^1$-homotopy property: the Gysin pullbacks $\lambda^!:A_{*+1}(Y\times\PP^1)\rightarrow A_*(Y)$ are independent of $\lambda\in\PP^1$.
So it is enough to prove that
\begin{align*}
p_{X*} \left( \ch^{Y_\infty}_{X_\infty}(\xi_\bullet|_\infty) \cap sp_A(-) \right) \ =\ p_{X*} \left( \ch^{X_\infty}_{X_\infty}(\xi_\bullet|_{X_\infty}) \cap sp\, (-) \right).
\end{align*}
For an integral closed substack $V \subset Y$, we obtain a decomposition
$$
sp_A[V] \ =\ sp\,[V] + [V_H], \ \ \ [V_H]\ \in\ A_{\dim V}(H \cap Y_\infty)
$$
by \eqref{sitH}. By \cite[Proposition 2.3(iii)]{PV:A}, we have 
$$
\ch^{Y_\infty}_{X_\infty}(\xi_\bullet|_\infty)\, \cap\, sp\,[V] \ =\ \ch^{X_\infty}_{X_\infty}(\xi_\bullet|_{X_\infty})\, \cap\, sp\,[V]
$$
since $sp\,[V] \in A_{\dim V}(X_\infty)$.
Let $\iota: H \cap X_\infty \hookrightarrow X_\infty$ be the closed immersion.
Then by \cite[Proposition 2.3(i),(iii)]{PV:A}, we have
\begin{align*}
\ch^{Y_\infty}_{X_\infty}(\xi_\bullet |_\infty)\, \cap\, [V_H] \ =\ \iota_* \left( \ch^{H \cap Y_\infty}_{H \cap X_\infty}(\xi_\bullet|_{H \cap Y_\infty})\, \cap\, [V_H] \right).
\end{align*}
Since $\xi_\bullet$ is strictly exact on $H$, the right-hand side is zero.
\end{proof}

\subsection{Proof of the multiplicative property (Theorem \ref{Multi:Prop})}
We follow the notations in Theorem \ref{Multi:Prop}.
Let $\xi^2_\bullet$ be a 2-periodic complex \eqref{xi2} constructed by using $E^2_\bullet$.
Similarly, we define $G^2$, $Y^2_\infty$, $X^2_\infty$, $sp^2_A$, $\widetilde{Y}^2$, $\widetilde{X}^2$ using $E^2_\bullet$.
Simply, we let $X := X_1 \cap X_2$, $X_\infty := X^2_\infty \times_{X_2} X$ and $\widetilde{X} := \widetilde{X}^2 \times_{X_2} X$.
\begin{Lemma} \label{Lem4}
We have an equality of homomorphisms
\begin{align*} 
\ch^Y_X(E^1_\bullet \ot E^2 _\bullet) \cap - \ =\ p_{X*} \left( \ch^{Y^2_\infty}_{X_\infty}(E^1_\bullet \ot \xi^2_\bullet) \cap sp^2_A(-) \right).
\end{align*}
Here, $E^1_\bullet$ on the right-hand side is the pullback along $Y^2_\infty\hookrightarrow G^2 \ra Y$.
\end{Lemma}

\begin{proof}
Note that the proof provides an alternative proof of Lemma \ref{rein} by letting 
$$E^1_\bullet := \cdots \ra \cO_Y \ra 0 \ra \cO_Y \ra 0 \ra \cdots.$$

The main idea we use is the homotopy property. 
We first consider a class 
$$
p_{\widetilde{X}*} \left( \ch^{\widetilde{Y}^2}_{\widetilde{X}}(E^1_\bullet\ot \xi^2_\bullet) \cap [\widetilde{V}^2] \right) \ \in\ A_{\dim V +1}\left(X\times (\PP^1\smallsetminus \{0\})\right)
$$
for an integral closed substack $V \subset Y$, where $p_{\widetilde{X}} : \widetilde{X} \ra X\times (\PP^1\smallsetminus \{0\})$ is the projection morphism.
Note that $E^1_\bullet \ot \xi^2_\bullet$ is strictly exact off $\widetilde{X}$ since $\xi^2_\bullet$ is strictly exact off $\widetilde{X}^2_\infty$.
The Gysin map
$$
\lambda^! \ :\ A_{\dim V +1}\left(X\times (\PP^1\smallsetminus \{0\})\right) \ \longrightarrow\ A_{\dim V}\left(X\right)
$$
is independent of the closed immersion $\lambda : \{\lambda\} \hookrightarrow \PP^1 \smallsetminus \{0\}$, for instance by \cite[Proposition 3.1.2]{Kr}. Hence we have
\begin{align} \label{spCial}
& 1^!\left( p_{\widetilde{X}*} \left( \ch^{\widetilde{Y}^2}_{\widetilde{X}}(E^1_\bullet\ot \xi^2_\bullet) \cap [\widetilde{V}^2] \right) \right) \\
& \qquad \qquad \qquad \qquad = \ \infty^!\left( p_{\widetilde{X}*} \left( \ch^{\widetilde{Y}^2}_{\widetilde{X}}(E^1_\bullet\ot \xi^2_\bullet) \cap [\widetilde{V}^2] \right) \right). \nonumber
\end{align}

Using the commutativity of refined Gysin pullbacks and pushforwards \cite[Theorem 3.12(i)]{Vis} and the bivariant property \cite[Definition 17.1(C3)]{Ful}, the left-hand side of \eqref{spCial} becomes
\begin{align*}
1^!\left( p_{\widetilde{X}*} \left( \ch^{\widetilde{Y}^2}_{\widetilde{X}}(E^1_\bullet\ot \xi^2_\bullet) \cap [\widetilde{V}^2] \right) \right) &= (p_{\widetilde{X}}|_1)_*1^!  \left(\ch^{\wY^2}_{\widetilde{X}}(E^1_\bullet \ot \xi^2_\bullet) \cap [\widetilde{V}^2] \right) \\
& = \ch^Y_X(E^1_\bullet \ot E^2 _\bullet) \cap 1^![\widetilde{V}^2] \\
& = \ch^Y_X(E^1_\bullet \ot E^2 _\bullet) \cap [V].
\end{align*}
Note that $p_{\widetilde{X}}|_1$ in the second term is an isomorphism, hence we omitted it in the last two terms.
On the other hand, the right-hand side of \eqref{spCial} becomes
\begin{align*}
\infty^!\left( p_{\widetilde{X}*} \left( \ch^{\widetilde{Y}^2}_{\widetilde{X}}(E^1_\bullet\ot \xi^2_\bullet) \cap [\widetilde{V}^2] \right) \right) &= p_{X*}\left( \infty^! \left( \ch^{\widetilde{Y}^2}_{\widetilde{X}}(E^1_\bullet\ot \xi^2_\bullet) \cap [\widetilde{V}^2] \right) \right) \\
& = p_{X*} \left( \ch^{Y^2_\infty}_{X_\infty}(E^1_\bullet \ot \xi^2 _\bullet|_\infty) \cap \infty^![\widetilde{V}^2] \right) \\
& = p_{X*} \left( \ch^{Y^2_\infty}_{X_\infty}(E^1_\bullet \ot \xi^2 _\bullet|_\infty) \cap sp^2_A([V]) \right).
\end{align*}
\end{proof}

Let $sp^2: A_*(Y) \ra B_*(X^2_\infty)$ be the specialization morphism defined in \eqref{mod:sp}.
The proof of the following lemma is parallel to that of Lemma \ref{rein}.

\begin{Lemma} \label{Lem5}
We obtain an equivalence
\begin{align*}
 p_{X*} \left( \ch^{Y^2_\infty}_{X_\infty}(E^1_\bullet \ot \xi^2_\bullet) \cap sp^2_A(-) \right)\ =\  p_{X*} \left( \ch^{X^2_{\infty}}_{X_\infty}(E^1_\bullet \ot \xi^2_\bullet) \cap sp^2 \,(-) \right).
\end{align*}
\end{Lemma} 
\begin{proof}
This is obvious because $E^1_\bullet \ot \xi^2_\bullet$ is strictly exact off $X_\infty \subset X^2_\infty$.
The rest of the proof is the same as the one of Lemma \ref{rein}.
\end{proof}

\subsection*{Proof of Theorem \ref{Multi:Prop}}
Let $p_{X_2}: X^2_\infty \ra X_2$ be the projection morphism. Then we obtain
\begin{align*}
p_{X*} \left( \ch^{X^2_\infty}_{X_\infty}(E^1_\bullet \ot \xi^2_\bullet) \cap sp^2(-) \right) & = p_{X*} \left( \ch^{X^2_\infty}_{X_\infty}(E^1_\bullet) \cdot \ch^{X^2_\infty}_{X^2_\infty}( \xi^2_\bullet) \cap sp^2(-) \right) \\
& = \ch^{X_2}_{X}(E^1_\bullet) \cdot  p_{{X_2}*} \left( \ch^{X^2_\infty}_{X^2_\infty}( \xi^2_\bullet) \cap sp^2(-) \right) \\
& = \ch^{X_2}_{X}(E^1_\bullet) \cdot  \ch^Y_{X_2}(E^2_\bullet) \cap -
\end{align*}
where the first equality comes from \cite[Proposition 2.3(v)]{PV:A}, the second equality is the commutativity of bivariant classes and pushforwards \cite[Definition 17.1(C1)]{Ful} and the third one comes from \eqref{X=Y} and \eqref{Alter:deC}.
Then the proof follows from Lemma \ref{Lem4} and \ref{Lem5}.

\section{Functoriality of cosection-localized intersection homomorphisms} \label{app}

In this section, we discuss an application to intersection theory.

\subsection{Koszul 2-periodic complex}\label{Sect:kos}
Let $Y$ be a finite type DM stack and $E$ be a vector bundle on $Y$.
Let $\alpha : \cO_Y \ra E^*$ be a section of $E^*$ and $\beta: \cO_Y \ra E$ be a section of $E$.
Suppose that the pairing $\lan \alpha , \beta \ran $ is zero.
The Koszul 2-periodic complex $\{\alpha, \beta\}$ is defined as
$$\{\alpha, \beta\}\ :=\ \{ \xymatrix{
\oplus_k \wedge^{2k} E^* \ar@<1ex>[r]  ^-{\wedge \alpha + \iota_{\beta}}
&  \ar@<1ex>[l]^-{\wedge \alpha + \iota_{\beta}}
\oplus_k \wedge^{2k+1} E^*
} \} .$$
It is a 2-periodic complex on $Y$ strictly exact off $Z(\alpha, \beta):=Z(\alpha) \cap Z(\beta)$ \cite[Lemma 2.2(2)]{KO1}. 
Here $\iota_{\beta}$ denotes the contraction by $\beta$ (see \cite[Section 2.2]{KO1} for the convention).
We sometimes consider $\alpha$ as a cosection of $E$ without any mention when the context is clear.

\subsection*{Set-up}
Consider the following diagram of vector bundles on $Y$
\begin{align*}
\xymatrix@R=6mm{
 & \cO_Y \ar[d]^-{\beta_1} & \cO_Y \ar[d]^-{\beta} & \cO_Y \ar[d]^-{\beta_2} & \\
 0 \ar[r] &E_1 \ar[d]^-{\alpha_1} \ar[r]^-f & E \ar[d]^-{\alpha} \ar[r]^-g  &E_2 \ar[d]^-{\alpha_2} \ar[r]  & 0 \\
 & \cO_Y& \cO_Y& \cO_Y&
}
\end{align*}
where the compositions of all verticals are zero and the horizontal is a short exact sequence.
Suppose that the following conditions are satisfied.
\begin{enumerate}
\item There is a section $\beta': \cO_Y \ra E$ such that $\beta_2 = g \circ \beta'$.
\item There is a cosection $\alpha' :E \ra \cO_Y$ such that $\alpha_1 = \alpha' \circ f$.
\item $\alpha' \circ \beta = \alpha \circ \beta'=0$ and $\alpha -\alpha' = \alpha_2 \circ g$.
\end{enumerate}

\begin{example} \label{example}
The above conditions are satisfied for $E=E_1 \oplus E_2$, $\alpha=\alpha_1 + \alpha_2$ and $\beta=(\beta_1,\beta_2)$.
\end{example}

With the above notations and conditions, we construct a vector bundle $\wE$ on $Y \times \A^1$
and sections 
$$
\wa\ :\ \cO_{Y \times \A^1} \ \longrightarrow\ \wE^* \text{  and  } \wb\ :\ \cO_{Y \times \A^1}\ \longrightarrow\ \wE$$ 
such that $\lan \wa, \wb \ran =0$ and
\begin{align*}
\{\wa, \wb\} |_{\lambda} = \left\{
\begin{array}{cl}
\{\alpha, \beta\} & \text{if } \lambda=1, \\
\{\alpha_1, \beta_1\} \ot \{\alpha_2, \beta_2\} & \text{if } \lambda=0
\end{array}
\right.
\end{align*}
where $\lambda$ is the coordinate on $\A^1$.
Consider the following morphisms of vector bundles on $Y \times \mathbb{A}^1$,
\begin{align*}
\xymatrix@R=6mm@C=20mm{
 & \cO_{Y \times \mathbb{A}^1} \ar[d]^-{((\lambda-1)\beta_1, \ \lambda \beta + (1-\lambda)\beta' )}\\
E_1 \ar[r]^-{(\lambda\, \text{id} , f)} & E_1 \oplus E \ar[d]^-{-\alpha_1 +  \left(\alpha + (\lambda -1)\alpha' \right)}  \\
 & \cO_{Y \times \mathbb{A}^1} .
}
\end{align*}
Here and the rest of the section, we omit the notations of pullbacks for bundles when they are clear from the context by abuse of notations.
Let $\wE$ be the cokernel of $(\lambda\, \text{id}, f)$. Since $f^*$ is surjective, so is $(\lambda\, \text{id}, f)^*$ whose kernel is $\wE^*$. Hence $\wE$ is a bundle over $Y\times\A^1$.
The section $\wb: \cO_{Y \times \mathbb{A}^1} \ra \wE$ is induced by $\left((\lambda-1)\beta_1, \lambda \beta + (1-\lambda)\beta' \right)$.
Since we have
\begin{align*}
\left(-\alpha_1 +  \left(\alpha +(\lambda-1)  \alpha' \right)\right) \circ (\lambda id, f)\ =\ \alpha f - \alpha' f\ =\ \alpha_2 \, g  f \ =\ 0,
\end{align*}
we obtain the induced cosection $\wa: \wE \ra \cO_{Y \times \mathbb{A}^1} $.
Since 
\begin{align*}
\alpha' \beta' \ =\ \alpha\, \beta' - \alpha_2\, g \beta'\ =\ -\alpha_2 \beta_2\ =\ 0,
\end{align*}
we have $\lan \wa, \wb \ran =0$.
The restriction of the short exact sequence
$$
\xymatrix@C=10mm{
0 \ar[r] & E_1 \ar[r]^-{(\lambda \, \text{id}, f)} & E_1 \oplus E \ar[r] &  \wE \ar[r] &0
}
$$
to $\lambda=1$ is
$$
\xymatrix@C=10mm{
0 \ar[r] & E_1 \ar[r]^-{(\text{id}, f)} & E_1 \oplus E \ar[r]^-{-f+\text{id}} &  E \ar[r] &0.
}
$$
Hence the restriction $\wb|_{\lambda=1}$ is $(-f+\text{id})\circ(0, \beta)=\beta$ and $\wa|_{\lambda=1}$ is $\alpha$ because 
$$
-\alpha_1 + \alpha - \left( - \alpha f + \alpha \right)\ =\ -\alpha_1 + \alpha f\ =\ \alpha_2\, g f\ =\ 0.
$$
At $\lambda=0$, the restriction is
$$
\xymatrix@C=10mm{
0 \ar[r] & E_1 \ar[r]^-{(0, f)} & E_1 \oplus E \ar[r]^-{(-\text{id},\, g)} &  E_1 \oplus E_2 \ar[r] &0.
}
$$
Thus, the restriction $\wb|_{\lambda=0}$ is $(-\text{id},\ g)\circ(-\beta_1, \beta')=(\beta_1,\beta_2)$ and $\wa|_{\lambda=0}$ is $\alpha_1 + \alpha_2$.

Let us further assume that the common zero is contained in
\begin{align}\label{Fass}
Z(\wa, \wb) \ \subset\ X \times \mathbb{A}^1
\end{align} 
for some closed substack $X \subset Y$.
For instance if $\beta- \beta' = f \circ \beta_1$, then $X$ can be taken to be $Z(\alpha, \beta)$ since it is the case of Example \ref{example} \emph{locally}.
By applying the homotopy property \cite[Lemma 2.1]{PV:A} to $\mathrm{ch}^{Y \times \mathbb{A}^1}_{X \times \mathbb{A}^1}(\{\wa, \wb\})$, we have the following multiplicative formula.

\begin{Lemma} \label{Lemma1}
With the above set-up and the assumption \eqref{Fass}, we obtain the following equalities
\begin{align} \label{multiapp}
\mathrm{ch}^Y_X(\{\alpha, \beta\})\ & =\ \mathrm{ch}^{X \cup Z(\alpha_1,\beta_1)}_X (\{\alpha_2, \beta_2\}) \circ \mathrm{ch}^Y_{X \cup Z(\alpha_1,\beta_1)} (\{\alpha_1, \beta_1\})  \\ \nonumber
& =\ \mathrm{ch}^{X \cup Z(\alpha_2,\beta_2)}_X (\{\alpha_1, \beta_1\}) \circ \mathrm{ch}^Y_{X \cup Z(\alpha_2,\beta_2)} (\{\alpha_2, \beta_2\}).
\end{align}
Here $X\cup Z(\alpha_i,\beta_i)$ is a closed substack of $Y$ so that it comes with the scheme structure.
\end{Lemma}

\begin{proof}
We have
\begin{align*}
\mathrm{ch}^Y_X(\{\alpha, \beta\})\ & =\ 1^! \mathrm{ch}^{Y \times \mathbb{A}^1}_{X \times \mathbb{A}^1}(\{\wa, \wb\}) \\
& =\ 0^! \mathrm{ch}^{Y \times \mathbb{A}^1}_{X \times \mathbb{A}^1}(\{ \wa ,  \wb  \}) \\
& =\ \mathrm{ch}^Y_X(\{\alpha_1, \beta_1\} \ot \{\alpha_2, \beta_2 \}) \\
& =\ \mathrm{ch}^{X \cup Z(\alpha_1,\beta_1)}_X (\{\alpha_2, \beta_2\}) \circ \mathrm{ch}^Y_{X \cup Z(\alpha_1,\beta_1)} (\{\alpha_1, \beta_1\})  .
\end{align*}
Here, the first and third equalities are the bivariant property of localized Chern characters \cite[Definition 17.1(C3)]{Ful}, the second is the homotopy property \cite[Lemma 2.1]{PV:A} and the fourth equality comes from Theorem \ref{Multi:Prop}.
This proves the first equality of \eqref{multiapp}. 
The proof for the second equality of \eqref{multiapp} is the same.
\end{proof}

\subsection{Cosection-localized intersection homomorphism}
Let $M$ be a finite type DM stack and $F$ be a vector bundle on $M$.
Let $\sigma: F \ra \cO_M$ be a cosection.
Recall that $w_\sigma$ is a function on $F$ induced by $\sigma$.
Let $p: F \ra M$ be the projection morphism and $\tau_F: \cO_F \ra p^*F$ be the tautological section induced by the diagonal morphism
$$
F\ \longrightarrow\ F \times_M F.
$$
Since
$w_\sigma = p^*\sigma \circ \tau_F \in \Gamma(\cO_{F})$,
the Koszul complex $\{p^*\sigma, \tau_F\}$ defines a 2-periodic complex on $Z(w_\sigma)$.
Moreover, the zero loci are $Z(p^*\sigma) = F|_{Z(\sigma)} \subset F$ and $Z(\tau_F) = M \subset F$ where the second one is given by the zero section.
Hence $Z(p^*\sigma, \tau_F) = Z(\sigma)$. Thus $\{p^* \sigma, \tau_F\}$ is strictly exact off $Z(\sigma)$.

As we have mentioned in the introduction, Kim and the author proved that the cosection-localized intersection homomorphism is expressed in terms of a localized Chern character \cite[Theorem 1.1]{KO1}
\begin{align} \label{KO}
0^!_{F, \sigma}\ =\ \mathrm{td}(F|_{Z(\sigma)}) 
\cdot \mathrm{ch}^{Z(w_\sigma)}_{Z(\sigma)}(\{p^*\sigma, \tau_F\}).
\end{align}
This immediately proves Theorem \ref{Cosec:Func} using Lemma \ref{Lemma1}.

Indeed, Theorem \ref{Cosec:Func} can be improved to the following general situation. We provide a proof for this instead of giving that of Theorem \ref{Cosec:Func}.
Consider the following diagram of vector bundles on $M$,
\begin{align*}
\xymatrix@R=6mm{
 0 \ar[r] &F_1 \ar[d]^-{\sigma_1} \ar[r]^-f & F \ar[d]^-{\sigma} \ar[r]^-g  &F_2 \ar[d]^-{\sigma_2} \ar[r]  & 0 \\
 & \cO_M& \cO_M& \cO_M&
}
\end{align*}
where the horizontal is a short exact sequence.
Suppose that there exists a cosection $\sigma' : F \ra \cO_M$ such that $\sigma_1 = \sigma' \circ f$ and $\sigma-\sigma' = \sigma_2 \circ g$.
Letting $Y:= Z(w_{\sigma_1}) \times_M Z(w_{\sigma_2})$ and $p: Y \ra M$ be the projection morphism, we obtain a diagram
\begin{align*}
\xymatrix@R=6mm{
 & \cO_Y \ar[d]^-{\tau_{F_1}} & \cO_Y \ar[d]^-{\tau_F} & \cO_Y \ar[d]^-{\tau_{F_2}} & \\
 0 \ar[r] &p^*F_1 \ar[d]^-{p^*\sigma_1} \ar[r]^-f & p^*F \ar[d]^-{p^*\sigma} \ar[r]^-g  &p^*F_2 \ar[d]^-{p^*\sigma_2} \ar[r]  & 0 \\
 & \cO_Y& \cO_Y& \cO_Y&
}
\end{align*}
in the set-up of Section \ref{Sect:kos}.
By letting $\tau' := \tau_F - f \circ \tau_{F_1}$, we can check that
\begin{align} \label{sevid}
g \circ \tau' = \tau_{F_2},\ \ \tau_F - \tau' = f \circ \tau_{F_1},\ \ p^*\sigma \circ \tau' =p^*\sigma' \circ \tau_F =0
\end{align}
using the \emph{local} splitting of the exact sequence.
The zero locus of $\sigma$ is $Z(\sigma) = Z(\sigma_1) \cap Z(\sigma_2)$.
Note that we may not have the last identity in \eqref{sevid} on $Z(w_\sigma)$ (which is bigger than $Y$).

\begin{Thm}
Let $F, F_1, F_2, \sigma, \sigma_1, \sigma_2, f, g$ and $\sigma'$ be as above such that $\sigma_1 = \sigma' \circ f$ and $\sigma- \sigma' = \sigma_2 \circ g$.
Then we have the following functorial property
\begin{align} \label{multiapp1}
0^!_{F, \sigma}\, =\, 0^!_{F_1, \sigma_1} \circ 0^!_{F_2, \sigma_2} &\, =\, 0^!_{F_2, \sigma_2} \circ 0^!_{F_1, \sigma_1} : A_*(Y) \ra A_{*-\rank F}(Z(\sigma)),
\end{align}
where $Y:=Z(w_{\sigma_1}) \times_M Z(w_{\sigma_2})$.
\end{Thm}
\begin{proof} We have
\begin{align*}
0^!_{F, \sigma}\ &=\ \mathrm{td}(F|_{Z(\sigma)}) 
\cdot \mathrm{ch}^{Y}_{Z(\sigma)}(\{p^*\sigma, \tau_F\}) \\
& =\ \mathrm{td}(F_1|_{Z(\sigma)}) \mathrm{td}(F_2|_{Z(\sigma)}) 
\cdot \mathrm{ch}^{Y_2}_{Z(\sigma)}(\{p^*\sigma_1, \tau_{F_1}\} ) \circ \ch^Y_{Y_2}( \{p^*\sigma_2, \tau_{F_2}\}) \\
& =\ \mathrm{td}(F_1|_{Z(\sigma)}) 
\cdot \mathrm{ch}^{Y_2}_{Z(\sigma)}(\{p^*\sigma_1, \tau_{F_1}\} ) \circ \mathrm{td}(F_2|_{Y_2})  \ch^Y_{Y_2}( \{p^*\sigma_2, \tau_{F_2}\}) \\
& =\ 0^!_{F_1, \sigma_1} \circ 0^!_{F_2, \sigma_2},
\end{align*}
where $Y_2 := Z(w_{\sigma_1}) \times_M Z(\sigma_2)$ and $F_2|_{Y_2}$ denotes the pullback of $F_2$ to $Y_2$.
Here the first and fourth equalities are from \eqref{KO}, the second one comes from Lemma \ref{Lemma1} and the third one is the commutativity of bivariant classes and Chern classes \cite[Proposition 17.3.2]{Ful}.
This proves the first equality of \eqref{multiapp1}. 
The proof for the second equality of \eqref{multiapp1} is the same.
\end{proof}

\end{document}